\documentclass[12pt]{amsart}

\usepackage{times}
\usepackage{amssymb}
\usepackage{enumitem}
\usepackage{tikz}

\usepackage[T1]{fontenc}
\usepackage[utf8]{inputenc}

\usepackage{hyperref}

\theoremstyle{plain}

\newtheorem{lemma}{Lemma}[section]
\newtheorem{theorem}[lemma]{Theorem}

\newtheorem{proposition}[lemma]{Proposition}

\theoremstyle{remark}

\newtheorem{definition}[lemma]{Definition}

\newcommand{\A}{\mathbb{A}}

\newcommand{\N}{\mathbb{N}}

\newcommand{\Sym}{\mathfrak{S}}

\newcommand{\Z}{\mathbb{Z}}

\DeclareMathOperator{\Tr}{Tr}

\author{Maciej Do\l\k{e}ga}
\address{Institute of Mathematics,
University of Wroclaw,  \mbox{pl.\ Grunwaldzki~2/4,} 50-384
Wroclaw, Poland}
\email{Maciej.Dolega@math.uni.wroc.pl}

\author{Piotr \'Sniady}
\address{Institute of Mathematics, Polish Academy of Sciences,
\mbox{ul.~Śniadeckich 8}, 00-956 Warszawa, Poland \newline
\indent Institute of Mathematics,
University of Wroclaw,  \mbox{pl.\ Grunwaldzki~2/4,} 50-384
Wroclaw, Poland}

\email{Piotr.Sniady@math.uni.wroc.pl}

\title[Structure of Kerov character polynomials]{Asymptotics of characters of
symmetric groups: structure of Kerov character polynomials}

\begin{document}

\begin{abstract}
We study asymptotics of characters of the symmetric groups on a fixed conjugacy class.
It was proved by Kerov that such a character can be expressed as a polynomial in free cumulants
of the Young diagram (certain functionals describing the shape of the Young
diagram).
We show that for each genus there exists a universal symmetric
polynomial which gives the coefficients of the part of Kerov character
polynomials with the prescribed homogeneous degree. The existence of such
symmetric polynomials was conjectured by Lassalle.
\end{abstract}

\maketitle

% \tableofcontents

\section{Introduction}
\subsection{Asymptotic representation theory of symmetric groups}
  
\emph{What can we say about the representations of the symmetric groups
$\Sym(n)$ in the limit $n\to\infty$?} This very general question is the subject
of investigations of \emph{the asymptotic representation theory of the symmetric
groups}. Even though for almost any question of the representation theory of the
symmetric groups the answer is known, usually this answer is given by a
combinatorial algorithm (for example, Murnaghan-Nakayama rule or
Littlewood-Richardson rule) involving manipulations with boxes of a Young
diagram. As $n$, the number of boxes, tends to infinity, such combinatorial
algorithms become very cumbersome and it is not easy to extract from them some
reasonable asymptotic answer. For this reason one has to look for new,
alternative approaches, which would less depend on the details of boxes of a
given Young diagram, but rather on its `global' features.

\subsection{Asymptotic shape of Young diagrams}
In this article we study the scaling of \emph{balanced Young diagrams} which
means that a Young diagram with $n$ boxes is assumed to have at most
$O(\sqrt{n})$ rows and columns. This scaling almost inevitably leads to the
concept of \emph{(asymptotic) shape} of a Young diagram: roughly speaking, we
disregard the information about the number of boxes of a Young diagram and we
are interested only how the Young diagram looks in large-scale perspective. More
precisely, this concept of (asymptotic) shape of a Young diagram corresponds to
the \emph{dilated Young diagram} $\frac{1}{\sqrt{n}} \lambda$ which, roughly
speaking, is obtained by replacing each unit box of a Young diagram by a box of
dimensions $\frac{1}{\sqrt{n}} \times \frac{1}{\sqrt{n}}$. Such a dilated Young
diagram is usually no longer a Young diagram but is a \emph{generalized Young
diagram} and we should not regard it as a combinatorial object but rather as a
geometric one. Since the area of the dilated Young diagram $\frac{1}{\sqrt{n}}
\lambda$ is always equal to $1$ (where $n$ denotes the number of boxes of
$\lambda$), this setup is very convenient for comparing shapes of Young diagrams
with different number of boxes. In this way we get a unified framework which
allows us to consider and compare Young diagrams with various number of boxes,
all at the same time. 

Probably the most celebrated result related to this scaling of Young diagrams is
the one of Logan and Shepp \cite{LoganShepp1977} and Ver{\v{s}}ik and Kerov
\cite{VervsikKerov1977} who proved that a random Young diagram (distributed
according to the Plancherel measure) will typically be very close to some
explicit asymptotic shape.

In order to keep this paper as simple as possible and to avoid generalized
Young diagrams, in the following we will consider only dilations of Young
diagrams by factors which are positive integers. This operation can be
easily described on a graphical representation of a Young diagram: we just
dilate the picture of $\lambda$ or, alternatively, we replace each box of
$\lambda$ by a grid of $s\times s$ boxes, see Figure \ref{fig:dilation}.
Note that if we fix a Young diagram $\lambda$ then the sequence of dilated Young
diagrams $(s\lambda)_{s=1,2,\dots}$ is an example of a collection of balanced
Young diagrams. It follows that our rather vague plan of studying balanced
Young diagrams can be made more concrete by studying the sequence of Young
diagrams $(s\lambda)_{s=1,2,\dots}$ in the limit as $s\to\infty$.

\begin{figure}[tb]
         \begin{tikzpicture}[scale=0.5]
            \begin{scope} 
                  \clip (0,2) 
                  -- (1,2) 
                  -- (1,1) 
                  -- (3,1) 
                  -- (3,0)
                  -- (0,0); 
                  \draw grid (4.3,3.4);
            \end{scope} 
 %             \draw[->] (-0.2,0) -- (4.7,0);
%             \draw[->] (0,-0.2) -- (0,3.7);
            \draw (1.5,-0.5) node[anchor=north] {$\lambda$}; 
            \draw[ultra thick] (0,2) 
            -- (1,2) 
            -- (1,1) 
            -- (3,1) 
            -- (3,0)
            -- (0,0) 
            -- cycle ;
        \end{tikzpicture}
\hspace{10ex}
%         \begin{tikzpicture}[scale=0.5]
%          \draw[white] (0,0) -- (0,0);  
%          \draw (0,3) node {$\mapsto$};
%         \end{tikzpicture}  
         \begin{tikzpicture}[scale=0.5]
            \begin{scope} 
                  \clip (0,6) 
                  -- (3,6) 
                  -- (3,3) 
                  -- (9,3) 
                  -- (9,0)
                  -- (0,0); 
                  \draw grid (12.3,9.4);
            \end{scope} 
 %             \draw[->] (-0.2,0) -- (4.7,0);
%             \draw[->] (0,-0.2) -- (0,3.7);
            \draw (4.5,-0.5) node[anchor=north] {$3\lambda$}; 
            \draw[ultra thick] (0,6) 
            -- (3,6) 
            -- (3,3) 
            -- (9,3) 
            -- (9,0)
            -- (0,0) 
            -- cycle ;
        \end{tikzpicture}
\caption{Young diagram $\lambda=(3,1)$ drawn in the french convention and its
dilation $3\lambda$.}
\label{fig:dilation}
\end{figure}
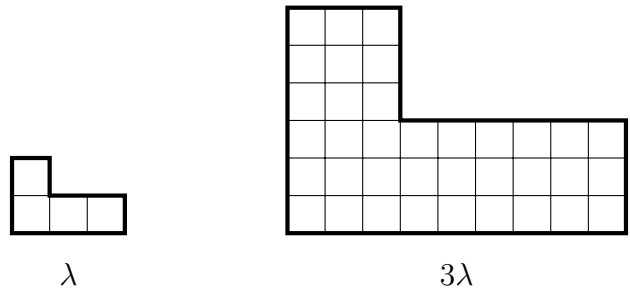

\subsection{How to normalize the characters?}
\emph{How should we normalize the characters of the symmetric groups in order to
obtain some meaningful asymptotic quantities?} 
The answer for this question was given by \cite{Biane2003} who
gave the following definition.
For any permutation $\pi\in \Sym(k)$ and an irreducible representation
$\rho^\lambda$ of the
symmetric group $\Sym(n)$ corresponding to the Young diagram $\lambda$ we
define the \emph{normalized character}
\begin{equation}
\label{eq:character}
 \Sigma^\lambda_\pi =\begin{cases} \underbrace{n(n-1)\cdots (n-k+1)}_{k \text{
factors}} \frac{\Tr \rho^\lambda(\pi)}{\text{dimension of $\rho^\lambda$}} &
\text{if } k\leq n, \\ 0 & \text{otherwise.} \end{cases} 
\end{equation}
An interesting feature of this definition is that we do not require that $k$,
the index of the symmetric group to which $\pi$ belongs, must be equal to $n$,
the number of boxes of $\lambda$. In order for $\rho^\lambda(\pi)$ to make
sense, for $k<n$ we just declare that the permutation $\pi\in\Sym(k)$ can be
also regarded as an element of $\Sym(n)$; we just add to $\pi$ additional $n-k$
fixpoints. As it was pointed out by Scarabotti \cite{Scarabotti2011}, it would
be more appropriate to use the name of \emph{spherical function} instead of
\emph{character} for these objects, nevertheless we will stick to this old
nomenclature.

Particularly interesting are the values of characters on cycles, therefore we
will use the notation
$$ \Sigma^\lambda_k = \Sigma^\lambda_{(1,2,\dots,k)}. $$
% where we treat the cycle $(1,2,\dots,k)$ as an element of $\Sym(k)$ for any
% integer
% $k\geq 1$.

In this article we will study the following problem: \emph{for fixed value of
$k$, what
can we say about the normalized characters $\Sigma^{\lambda}_k$ when a balanced
Young diagram $\lambda$ tends to infinity} or, in a slightly more concrete
reformulation, \emph{what can we say about the normalized characters $\Sigma^{s
\lambda}_k$ related to dilated Young diagrams $s\lambda$ in the limit as
$s\to\infty$?}

\subsection{How to describe the shape of a Young diagram?}
A natural question arises: \emph{how to choose parameters which describe the
shape of a Young diagram in the most convenient way?} A very interesting answer
for this question was proposed by Biane \cite{Biane1998} who for a (generalized)
Young diagram $\lambda$ defined a family of parameters
$R_2^{\lambda},R_3^{\lambda},\dots$, called \emph{free cumulants} of $\lambda$.

The original definition of free cumulants of $\lambda$ given by Biane is quite
involved (free cumulants of a Young diagram are Speicher's \emph{free cumulants}
\cite{NicaSpeicher2006} --- related to Voiculescu's \emph{free probability
theory} \cite{VoiculescuDykemaNica1992} --- of Kerov's \emph{transition
measure} \cite{Kerov1993} of the Young diagram $\lambda$), but it has an
advantage that it is very explicit and allows an algorithmic computation of
free cumulants in terms of the \emph{shape} of a Young diagram.

From the results of Biane \cite{Biane1998} one can show the very surprising fact
that for any integer $k\geq 1$ and any Young diagram $\lambda$ the values of the
normalized character on dilations of $\lambda$
$$ \mathbb{N} \ni s \mapsto \Sigma_k^{s\lambda} $$
are given by a polynomial function of degree (at most) $k+1$. Furthermore, the
leading coefficient is equal to one of the free cumulants:
\begin{equation} 
\label{eq:biane}
R_{k+1}^{\lambda}=[s^{k+1}] \Sigma_k^{s\lambda} =
\lim_{s\to\infty} \frac{\Sigma^{s\lambda}_{k}}{s^{k+1}}.
\end{equation}
From the perspective of the Biane's work \cite{Biane1998} this is a highly
nontrivial and very interesting result: it shows that free cumulants (which are
viewed as concrete, algorithmically computable quantities) describe the
first-order asymptotics of characters.
For the purposes of this article we can reverse the optics and take
\eqref{eq:biane} as a convenient (even if somewhat abstract) definition of free
cumulants.

\subsection{Free cumulants}
% Let $\lambda$ be a Young diagram. 
We define the \emph{free
cumulants} $R_2^\lambda,R_3^\lambda,\dots$ as
% \begin{equation}
% \label{eq:free-cumulants-strange}
$$R_{k}^\lambda = \lim_{s\to\infty} \frac{1}{s^{k}}
\Sigma^{s\lambda}_{k-1}, $$
% \end{equation}
in other words each free cumulant is asymptotically the dominant term of the
character on a cycle of appropriate length in the limit when the Young diagram
tends
% in some sense
to infinity.

One of the reasons why free cumulants are so useful in asymptotic
representation theory is that they are homogeneous with respect to dilations of
the Young diagrams, namely
$$ R^{s\lambda}_k = s^k R^{\lambda}_k.$$
% in other words the degree of the free cumulant $R_k$ is equal to $k$.

% In fact, the notion of free cumulants origins from the work of Voiculescu
% \cite{Voiculescu1986} where they appeared as coefficients of an $R$-series
% which turned out to be useful in description of \emph{free convolution} in the
% context of \emph{free probability theory} \cite{VoiculescuDykemaNica1992}. The
% name of free cumulants was coined by Speicher \cite{Speicher1998} who found
% their combinatorial interpretation and their relations with the lattice of
% non-crossing partitions \cite{Speicher1993}. Since free probability theory is
% closely related to the random matrix theory \cite{Voiculescu1991} free
% cumulants quickly became an important tool not only within the framework of
% free probability but in the random matrix theory as well.

\subsection{Kerov character polynomials}
It turns out that free cumulants can be used not only to provide asymptotic
approximations for the characters of symmetric groups, but also for exact
formulas. Kerov during a talk in Institut Henri Poincar\'e in January 2000
\cite{Kerov2000talk} announced the following result (the first published proof
was given by Biane \cite{Biane2003}): for each permutation $\pi$ there exists a
unique universal polynomial $K_{\pi}$ with integer coefficients, called
\emph{Kerov character polynomial}, with a property that
% \begin{equation}
% \label{eq:kerov-polynomial}
$$\Sigma^{\lambda}_\pi =
K_{\pi}(R_2^{\lambda},R_3^\lambda,\dots)$$
% \end{equation}
holds true for any
% (generalized)
Young diagram $\lambda$.
We say that the Kerov polynomial is universal because it does not depend on the
choice of $\lambda$. In order to simplify notation we suppress the
$\lambda$-dependence of characters and free cumulants, writing
$$ \Sigma_\pi = K_{\pi}(R_2,R_3,\dots).$$

As usual, we are mostly concerned with the values of the characters on cycles,
therefore we introduce special notation for such Kerov polynomials
$$ \Sigma_k = K_{k}(R_2,R_3,\dots).$$
% Kerov also found the leading term of the Kerov polynomial:
% \begin{equation}
% \label{eq:expansion}
% \Sigma_k = R_{k+1} + (\text{terms of degree at most $k-1$})
% \end{equation}
% which has \eqref{eq:free-cumulants-strange} as an immediate consequence.
The first few Kerov polynomials $K_k$ are as follows \cite{Biane2001}:
\begin{align*}
\Sigma_1 &= R_2, \\
\Sigma_2 &= R_3, \\
\Sigma_3 &= R_4 + R_2,   \\
\Sigma_4 &= R_5 + 5R_3,    \\
\Sigma_5 &= R_6 + 15R_4 + 5R_2^2 + 8R_2, \\
\Sigma_6 &= R_7 + 35R_5 + 35R_3 R_2 + 84R_3.
\end{align*}

The primary motivation for investigation of this subject is the asymptotic
representation theory, namely a good understanding of Kerov character
polynomials might in the future shed some light on asymptotics of characters,
also in the most difficult scaling when the length of the permutation on which
we evaluate the character grows with the number of boxes of the Young diagram,
see \cite{F'eray'Sniady2011}. 

The second motivation --- which is key for the purposes of this article --- is
related to \emph{algebraic combinatorics}. The last decade has seen a number of
research papers which stated (sometimes conjecturally) several very surprising
combinatorial properties of Kerov polynomials. For example, it was conjectured
by Kerov \cite{Kerov2000talk} that the coefficients of Kerov polynomials are
non-negative integers. These papers showed not only the richness of the
combinatorial and the analytic structures of the Kerov character polynomials
but also the difficulty in fully understanding these polynomials. Since the
results proved in most of these papers will be necessary for the purposes of
this article we decided to postpone the presentation of these papers until they
are needed. A more complete presentation of the history of the subject and
bibliography can be found in the paper \cite{DolegaF'eray'Sniady2008}.

\subsection{Genus expansion}
It is convenient to consider a gradation with respect
to which the degree of the free cumulant $R_k$ is equal to $k$. We denote by
$K_{k,d}$ the homogeneous part of degree $d$ of the Kerov character
polynomial $K_k$. One of the results announced by Kerov \cite{Kerov2000talk}
was that the only non-zero polynomials $K_{k,d}$ are of the form
$K_{k,k+1-2g}$ where $g\geq 0$ is an integer. It is possible to give some
topological meaning to many calculations related to Kerov polynomials in which
the integer $g$ can be interpreted as the genus of the resulting two-dimensional
surface. For this reason, studying the polynomials $K_{k,k+1-2g}$ for a fixed
value of $g$ is often called the \emph{genus expansion}.

The form of the highest-degree term
$$ K_{k,k+1} = R_{k+1} $$
was announced by Kerov \cite{Kerov2000talk} and proved by Biane
\cite{Biane2003}. The form of the next term $K_{k,k-1}$ was conjectured by
Biane \cite{Biane2003} and proved by Śniady \cite{'Sniady2006a}.
Explicit but rather complicated formulas for the general genus $K_{k,k+1-2g}$
were found by Goulden and Rattan \cite{GouldenRattan2007} (for a more
elementary proof we refer to the work of Biane \cite{Biane2005/07})
and we shall discuss their result in Section \ref{sec:Goulden-Rattan}.

\subsection{The main result: proof of some conjectures of Lassalle}
Lassalle announced as a conjecture \cite{Lassalle-preprint2007} that there is
an additional structure in the genus expansion of Kerov polynomials. He claimed
that for a fixed genus $g$ there exists a symmetric function $f_g$ which
describes polynomials $K_{k,k+1-2g}$ and which is independent of $k$. Before
presenting his conjectures we need to prepare some notations.
 
A partition $\mu=(\mu_1\geq \mu_2\geq \cdots)$ is a weakly decreasing
sequence of nonnegative integers with finitely many non-zero elements. The
non-zero $\mu_i$ in a partition $\mu$ are
called the parts of $\mu$. We will denote by
$m_i(\mu)$ the number of parts of $\mu$ equal to $i$; by $l(\mu)$ the number
of parts of $\mu$; and denote $|\mu|=\mu_1+\mu_2+\cdots$. Following Lassalle,
we define $R_1 = 0$ and for a strictly positive integer $i$ we define
\begin{align}
\nonumber
\mathcal{R}_i &= (i-1)R_i, \\
\nonumber
\mathcal{R}_\mu &= \prod_{i}
\frac{\mathcal{R}_i^{m_i(\mu)}}{m_i(\mu)!}, \\
\label{eq:qi}
Q_i &= \sum_{|\mu|=i}\big(l(\mu)-1\big)!\ \mathcal{R}_\mu, \\
\nonumber
\mathcal{Q}_\mu &= \prod_{i} \frac{Q_i^{m_i(\mu)}}{m_i(\mu)!}.
\end{align}

As usual, we denote by $e_i$ the elementary symmetric functions, by $h_i$ the
complete symmetric functions and by $p_i$ the power-sum symmetric functions. For
any partition $\mu$, we denote by $e_{\mu}$, $h_{\mu}$ or $p_{\mu}$ their
product over the parts of $\mu$, and by $m_{\mu}$ the monomial symmetric
function --- the sum of all distinct monomials whose exponent is
a permutation of $\mu$.

The main result of this article is a proof of the following results
which were stated as the first and the sixth conjecture in the
paper \cite{Lassalle-preprint2007} by Lassalle.
\begin{theorem}
\label{theo:main}
For any $g\geq 1$ there exist inhomogeneous symmetric
functions $f_g$ and $h_g$, having maximal degree $4(g-1)$, such that
\begin{align}
\label{eq:lassalle1}
K_{k,k+1-2g} &= \binom{k+1}{3}  \sum_{|\mu|=k+1-2g}
\big(l(\mu) + 2g - 2\big)!\ f_g(\mu)\ \mathcal{R}_\mu \\
\label{eq:lassalle6}
 &= \binom{k+1}{3}  \sum_{|\mu|=k+1-2g}
(2g - 1)^{l(\mu)}\ h_g(\mu)\ \mathcal{Q}_\mu,
\end{align}
where $f_g(\mu) = f_g(\mu_1,\mu_2,\dots)$ and $h_g(\mu_1,\mu_2,\dots)$. These
symmetric functions are independent of $k$.
\end{theorem}

This result sheds some light on the structure of Kerov polynomials but it
also leads to many new open problems, in particular the positivity
conjectures of
Lassalle \cite{Lassalle-preprint2007} and his questions concerning combinatorial
interpretations of the coefficients in the expansions of the above symmetric
functions.

\subsection{General idea of the proof}
For a given positive integer $g$ we define a symmetric function
$k(\mu) := |\mu|+2g-1 = p_1(\mu) + 2g - 1$,
therefore for a given symmetric functions $f_g$ and $h_g$ we can define
symmetric functions $ \tilde{f}_g(\mu) := k(\mu) \binom{k(\mu)+1}{3} f_g(\mu)$
and $
\tilde{h}_g(\mu) := k(\mu) \binom{k(\mu)+1}{3} h_g(\mu)$. We notice that in
equations \eqref{eq:lassalle1} and \eqref{eq:lassalle6} we sum over
all partitions $\mu$ which satisfy $k(\mu) = k$.
Therefore the following proposition
is an immediate consequence of Theorem \ref{theo:main}.

\begin{proposition}
\label{prop:main-prop}
For any $g\geq 1$ there exist inhomogeneous symmetric
functions $\tilde{f}_g$ and $\tilde{h}_g$, having maximal degree $4g$, such
that
\begin{align*}
k\ K_{k,k+1-2g} &=  \sum_{|\mu|=k+1-2g}
\big(l(\mu) + 2g - 2\big)! \ \tilde{f}_g(\mu) \mathcal{R}_\mu \\ 
 &= \sum_{|\mu|=k+1-2g}
(2g - 1)^{l(\mu)}\ \tilde{h}_g(\mu) \mathcal{Q}_\mu,
\end{align*}
where $\tilde{f}_g(\mu) = \tilde{f}_g(\mu_1,\mu_2,\dots)$ and $\tilde{h}_g(\mu)
= \tilde{h}_g(\mu_1,\mu_2,\dots)$. These symmetric functions are independent of
$k$.
\end{proposition}

In fact, the opposite implication holds true as well and Theorem
\ref{theo:main} is a consequence of Proposition \ref{prop:main-prop}: roughly
speaking we will show that the symmetric functions $\tilde{f}_g$ and
$\tilde{h}_g$ are divisible by the polynomial $k \binom{k+1}{3}$. One can
notice that Proposition \ref{prop:main-prop} stated that $\tilde{f}_g$ and
$\tilde{h}_g$ are independent of $k$, so the divisibility which  we will show is
a divisibility of symmetric functions $\tilde{f}_g$ and
$\tilde{h}_g$ by the symmetric function $K$ which has a property that for any
partition $\mu$ such that $|\mu| = k +1 - 2g$ we have that $K(\mu) = k
\binom{k+1}{3}$. We shall explain precisely
what this divisibility means and
show in Section \ref{sec:divisibility} that it holds indeed by
studying the arithmetic properties of Kerov polynomials and their divisibility
by prime numbers.

The remaining difficulty is to prove Proposition \ref{prop:main-prop}. We shall
do it in Section \ref{sec:existence} by analysis of the Goulden-Rattan formula.

Section \ref{sec:technical} is a presentation of technical and complicated
proofs of lemmas which are used in the previous sections.

%QUESTION: CAN WE FIND ANOTHER PROOF OF Proposition \ref{prop:main-prop} USIN
%DIRECTLY THE RESULTS OF \cite{DolegaF'eray'Sniady2008}?

\section{Goulden-Rattan formula and existence of symmetric polynomials}
\label{sec:existence}

\subsection{Power series $P_\lambda$}
Following Goulden and Rattan \cite{GouldenRattan2007} we define
% \begin{equation}\label{Cdef}
\begin{equation} C(t)=\frac{1}{1-\sum_{i\geq 2}\mathcal{R}_i t^i} = 
% \sum_{k\geq
% 0} t^k \sum_{\substack{i_1,\dots,i_m\geq 2 \\ i_1+\cdots+i_m=k }}
% \mathcal{R}_{i_1} \cdots \mathcal{R}_{i_m} = \\
\sum_{\mu} t^{|\mu|} l(\mu)!\ \mathcal{R}_\mu.
\end{equation}
Let $D=t \frac{d}{dt}$
% , and $I$ be the identity operator,
and define for $m\geq 1$
% \begin{equation}
% \label{eq:wielomiany-P}
$$
P_m(t)=-\frac{1}{m!}C(t)\big(D+m-2\big)C(t)\cdots
\big(D+1\big)C(t)DC(t).$$
% \end{equation}
For example, we have
\begin{align}
\nonumber P_1(t)&=-C(t),\\ 
\nonumber P_2(t)&=-\frac{1}{2}C(t)DC(t),\\
\nonumber
% \label{P3}
P_3(t)&
=-\frac{1}{6}
C(t)(D+1)C(t)DC(t) \\
\nonumber &= -\frac{1}{6}\left[C(t)DC(t)DC(t) + C(t)^2DC(t)\right] \\
\nonumber &= 
-\frac{1}{6}
\left[C(t)D\big(C(t)\cdot DC(t)\big) + C(t)^2DC(t)\right] \\
\nonumber &=
-\frac{1}{6}\left[C(t)\big(DC(t)\big)^2+C(t)^2D^2C(t)+C(t)^2DC(t)\right] .
\end{align}
Finally, for a partition $\lambda$, we
write $P_{\lambda}(t)=\prod_{j=1}^{l(\lambda)} P_{\lambda_j}(t)$.

For $p = (p_0,\dots,p_l) \in \N^{l+1}$ we define
% \begin{equation}
% \label{eq:produkt}
$$ E(p) := C(t)^{p_0} D C(t)^{p_1} \cdots D C(t)^{p_l}. $$
% \end{equation}

\begin{lemma}
\label{lem:admissible}
\
\begin{enumerate}[label=(\alph*)]
 \item \label{admissible:A}
Let $p \in \N^{l+1}, q \in \N^{m+1}$ and denote by $|p| := p_0 + \cdots + p_l$
($|q| = q_0 + \cdots q_m$ respectively). Then
$$E(p)\cdot E(q) = \sum_{\substack{r \in \N^{l+m+1}, \\ |r| = |p|+|q|}}c^{p,q}_r
E(r),$$
where $c^{p,q}_r \in \Z$.
\item \label{admissible:B}
For any partition $\lambda$ the fraction $\frac{P_{\lambda}(t)}{C(t)}$
is a linear combination of terms $E(p)$ where $p \in \N^k$ such that $|p| =
|\lambda|-1$ and $k \leq| \lambda|-l(\lambda)+1$.
\end{enumerate}
\end{lemma}
\begin{proof}
% Our goal is to investigate the product
% % \begin{equation}
% % \label{eq:it-hurts}
% $$ \left[ C(t)^{p'_0} D  C(t)^{p'_1} \cdots D C(t)^{p'_{l'} }
% \right] \cdot
% \left[ C(t)^{p''_0} D  C(t)^{p''_1} \cdots D C(t)^{p''_{l''} }
% \right] .$$
% % \end{equation}
Let $p = (p_0,\dots,p_l) \in \N^{l+1}$, $q = (q_0,\dots,q_m) \in \N^{m+1}$.
We will show part \ref{admissible:A} by induction on $l$. It is obvious for
$l = 0$. For $l = 1$
% and for $l' = 3$ 
we have:
$$E(p)\cdot E(q) = E(p_0,p_1 + q_0, q_1,q_2,\dots,q_m) -
E(p_0+p_1,q_0,q_1,\dots,q_m)$$
% 
% \begin{multline*}
% \left[ C(t)^{p'_0} D  C(t)^{p'_1} \right] \cdot \left[ C(t)^{p''_0} D
% C(t)^{p''_1} \cdots D C(t)^{p''_{l''} } \right] = \\ C(t)^{p'_0} D  C(t)^{p'_1
% +
% p''_0} D C(t)^{p''_1} \cdots D C(t)^{p''_{l''} } - \\
% C(t)^{p'_0+p'_1} D  C(t)^{p''_0} D C(t)^{p''_1} \cdots D C(t)^{p''_{l''}}
% \end{multline*}
% and
% \begin{multline}
% \left[ C(t)^{p'_1} D  C(t)^{p'_2} D C(t)^{p'_3} \right] \cdot \left[
% C(t)^{p''_1} D 
% C(t)^{p''_2} \cdots D C(t)^{p''_{l''} } \right] = \\ C(t)^{p'_1} D 
% C(t)^{p'_2} D C(t)^{p'_3 + p''_1} D 
% C(t)^{p''_2} \cdots D C(t)^{p''_{l''} } - \\ 
% C(t)^{p'_1 + p'_2} D 
% C(t)^{p'_3} D C(t)^{p''_1} D 
% C(t)^{p''_2} \cdots D C(t)^{p''_{l''} } - \\ 
% C(t)^{p'_1} D 
% C(t)^{p'_2 + p'_3} D C(t)^{p''_1} D 
% C(t)^{p''_2} \cdots D C(t)^{p''_{l''} } + \\ 
% C(t)^{p'_1 + p'_2 + p'_3}
% D^2 C(t)^{p''_1} D 
% C(t)^{p''_2} \cdots D C(t)^{p''_{l''} }
% \end{multline}
by the Leibniz rule.
Let us
assume, that the inductive assertion holds for some $l \geq 1$ and let $p =
(p_0, \dots,p_{l+1})$. Then by the
Leibniz rule we have that
\begin{equation}
\label{leibnizA}
E(p)\cdot E(q) = C(t)^{p_0}D \left[ E(p')\cdot E(q)\right] -
C(t)^{p_0+p_1}\left[ E(p'')\cdot E(q')\right],
\end{equation}
where $p' = (p_1,\dots,p_{l+1})$, $p'' = (0,p_2,\dots,p_{l+1})$, $q' =
(0,q_0,\dots,q_m)$.
% \begin{multline}
% \label{leibnizA}
% \left[ C(t)^{p'_0} D  C(t)^{p'_1} \cdots D C(t)^{p'_{l'+1} }
% \right] \cdot
% \left[ C(t)^{p''_0} D  C(t)^{p''_1} \cdots D C(t)^{p''_{l''} }
% \right] = \\
%  C(t)^{p'_0} D \bigg( \left[  C(t)^{p'_1} D C(t)^{p'_2}
% \cdots D
% C(t)^{p'_{l'+1} }
% \right] \cdot
% \left[ C(t)^{p''_0} D  C(t)^{p''_1} \cdots D C(t)^{p''_{l''} }
% \right] \bigg) - \\ 
% \left[ C(t)^{p'_0 + p'_1} \right] \cdot \left[ D C(t)^{p'_2} \cdots D
% C(t)^{p'_{l'+1} }
% \right] \cdot
% \left[ D C(t)^{p''_0} D  C(t)^{p''_1} \cdots D C(t)^{p''_{l''} }
% \right] = \\
% C(t)^{p'_0} D \bigg( \sum_{\alpha}A_{\alpha} \bigg) - \sum_{\beta}B_{\beta},
% \end{multline}
By the inductive assertion, the right hand side of \eqref{leibnizA} is equal to
$$\sum_{\substack{\alpha \in \N^{l+m+1}, \\ |\alpha| = |p|+|q|-p_0}} c_\alpha
C(t)^{p_0} D
E(\alpha) - \sum_{\substack{ \beta \in \N^{l+m+2}, \\|\beta| = |p|+|q|}}c_\beta
E(\beta),$$ 
where $c_\alpha,c_\beta \in \Z$. But it means that
$$E(p)\cdot E(q) = \sum_{\substack{r \in \N^{l+m+2}, \\|r| = |p|+|q|}} c^{p,q}_r
E(r),$$
where $c^{p,q}_r \in \Z$ which finishes the proof of part \ref{admissible:A}.
% $A_{\alpha}$ is an expression of the form \eqref{eq:produkt} for each
% $\alpha$, $B_{\beta}$ is an expression of the form \eqref{eq:produkt} for
% each
% $\beta$ and the last equality in \eqref{leibnizA} holds by the inductive
% assertion.
% Hence
% \begin{multline*}
% \label{leibnizB}
% \left[ C(t)^{p'_0} D  C(t)^{p'_1} \cdots D C(t)^{p'_{l'+1} }
% \right] \cdot
% \left[ C(t)^{p''_0} D  C(t)^{p''_1} \cdots D C(t)^{p''_{l''} }
% \right] = \\
% \sum_{\alpha}C(t)^{p'_0} D A_{\alpha} - \sum_{\beta}B_{\beta}
% \end{multline*}
% which finishes the proof of part \ref{admissible:A}.

For part \ref{admissible:B} we notice that each function $P_m(t)$ is a linear
combination of $E(p)$, where $p \in \N^k, |p| = m$ and $k \leq m-1$ and we apply
part
\ref{admissible:A}.
\end{proof}

\subsection{Polynomial structure of coefficients of $P_\lambda$}

\begin{definition}
If $f$ is a symmetric function of degree $d$ and $2g\geq 2$ is an integer then
the formal power series
\begin{equation}
\label{eq:postac-F-nice}
F(t)=  \sum_{\mu}  t^{|\mu|}
\big(l(\mu) + 2g - 2\big)!\ f(\mu)\ \mathcal{R}_\mu
\end{equation}
will be called \emph{a power-sum of the
first kind with degree $d$ and genus $g$}
and the formal power series
\begin{equation}
\label{eq:postac-F-vnice}
F(t)=  \sum_{\mu}  t^{|\mu|}(2g-1)^{l(\mu)}\ f(\mu)\ \mathcal{Q}_\mu
\end{equation}
will be called \emph{a power-sum of the second kind with degree $d$ and genus
$g$}.
\end{definition}

\begin{lemma}
\label{lem:nice}
\
\begin{enumerate}[label=(\alph*)]
  \item \label{case1}
  $C(t)$ is a power-sum of the
first (respectively, second) kind with degree $0$ and genus $1$.
%  corresponding to $f(\mu)=1$,
  \item \label{case1.5}
If $F(t)$ is a  power-sum of the
first (respectively, second) kind with degree $d$ and
genus $g$ then $D F(t)$ is a power-sum of the
first (respectively, second) kind with degree
$d+1$ and genus $g$.
  \item \label{case2}
If $F(t)$ is a power-sum of the
first (respectively, second) kind with degree $d$ and genus
$g$ then $C(t) F(t)$
is a power-sum of the
first (respectively, second) kind with degree $d$ and genus
$g+\frac{1}{2}$.
\end{enumerate}
\end{lemma}
\begin{proof}
In order to prove point \ref{case1} it suffices to notice that
$$C(t)=\sum_{\mu}  t^{|\mu|} l(\mu)!\ 
\mathcal{R}_\mu = \sum_{\mu}  t^{|\mu|}\mathcal{Q}_\mu.$$

In order to prove point \ref{case1.5} let $F(t)$ be in the form
\eqref{eq:postac-F-nice}. Then
$$
DF(t)=\sum_{\mu}  t^{|\mu|}
\big(l(\mu) + 2g - 2\big)!\ \big[ (p_1(\mu) f(\mu)\big] \
\mathcal{R}_\mu
$$
is again of the form \eqref{eq:postac-F-nice}.

If $F(t)$ is of the form \eqref{eq:postac-F-vnice}, then
$$
DF(t)=\sum_{\mu}  t^{|\mu|}
(2g-1)^{l(\mu)}\ \big[ (p_1(\mu) f(\mu)\big] \
\mathcal{Q}_\mu
$$
is again of the form \eqref{eq:postac-F-vnice} which shows part \ref{case1.5}.

For part \ref{case2} we can
assume that the symmetric function $f$ is
equal to monomial symmetric function $m_{\lambda}$ for some partition
$\lambda$.
% Since any symmetric function
% is
% a linear combination of monomial symmetric functions, this assumption does not
% restrict the generality of
% our considerations.

We define
\begin{align*} C_n & = \sum_{|\mu|=n} l(\mu)!\ \mathcal{R}_\mu = \sum_{|\mu|=n}
\mathcal{Q}_\mu, \\
\quad
\mathcal{C}_\mu & =\prod_{i \ge 2} \frac{C_i^{m_i(\mu)}}{m_i(\mu)!}.
\end{align*}
The correspondence between these three families ($Q, R$ and $C$) is given by
\begin{equation*}
\begin{split}
Q_n&= \sum_{|\mu|=n} {(-1)}^{l(\mu)} (l(\mu)-1)!  \,  \mathcal{C}_\mu,\\
-\mathcal{R}_n&=\sum_{|\mu|=n} {(-1)}^{l(\mu)} \,  \mathcal{Q}_\mu=
\sum_{|\mu|=n} {(-1)}^{l(\mu)} l(\mu)!\,  \mathcal{C}_\mu.
\end{split}
\end{equation*}  
Following Lassalle, \cite{Lassalle-preprint2007} we define the
(formal) alphabet
$\A$ by
$$ \mathcal{R}_i = - h_i(\A), \quad Q_i = - p_i(\A) / i, \quad
C_i = (-1)^{i} e_i(\A).$$
Writing
$$ u_{\mu} = l(\mu)! / \prod_{i \geq 1} m_i(\mu), \quad \epsilon_{\mu} =
(-1)^{n - l(\mu)}, \quad z_{\mu} = \prod_{i \geq 1} i^{m_i(\mu)}m_i(\mu)!, $$
the previous relations can be understood in a frame of symmetric functions
theory, and they are merely the classical properties~\cite[pp.~25 and
33]{Macdonald1995}
\begin{equation*}
\begin{split}
p_n&=-n \sum_{|\mu|=n} {(-1)}^{l(\mu)} \,  u_\mu h_\mu/l(\mu)
=-n \sum_{|\mu|=n} \epsilon_\mu  u_\mu e_\mu/ l(\mu),\\
e_n&= \sum_{|\mu|=n} \epsilon_\mu u_\mu h_\mu= \sum_{|\mu|=n} \epsilon_\mu
z_{\mu}^{-1} p_{\mu},\\
h_n&=\sum_{|\mu|=n} z_{\mu}^{-1} p_{\mu}=
\sum_{|\mu|=n} \epsilon_\mu u_\mu e_\mu.
\end{split}
\end{equation*}

% \textbf{ Jeśli dobrze rozumiem, twierdzisz, że interesujący nas power-sum of
% the first kind daje się zapisać w postaci
% $$\bigg(\sum_{\mu} t^{|\mu|} m_{\lambda}(\mu) \frac{(l(\mu) + 2g -
% 2)!}{l(\mu)!} (-1)^{l(\mu)} u_{\mu} h_{\mu} \bigg) $$,
% zaś interesujący nas power-sum of the second kind daje się zapisać w postaci 
% $$ \bigg(\sum_{\mu} t^{|\mu|} m_{\lambda}(\mu) (2g - 1)^{l(\mu)}
% (-1)^{l(\mu)} z^{-1}_{\mu} p_{\mu} \bigg). $$
% Ja tego chyba nie widzę, ale to pewnie jest związane z tym, że nie rozumiem
% alfabetów. Tym  nie mniej może warto
% poinformować czytelnika o tym.
% }

Using this notation, it suffices to show
that
for any monomial symmetric function $m_{\lambda}$ and any $g$ we have
\begin{multline}
\label{sym:R}
\bigg(\sum_{\mu} t^{|\mu|} m_{\lambda}(\mu) \frac{(l(\mu) + 2g -
2)!}{l(\mu)!} (-1)^{l(\mu)} u_{\mu} h_{\mu} \bigg) \bigg(\sum_{\rho} t^{|\rho|}
(-1)^{l(\rho)} u_{\rho} h_{\rho} \bigg) = \\
\bigg(\sum_{\nu} t^{|\nu|} \frac{m_{\lambda}(\nu)}{l(\lambda) + 2g -1}
\frac{(l(\nu) + 2g -
1)!}{l(\nu)!} (-1)^{l(\nu)} u_{\nu} h_{\nu} \bigg),
\end{multline}
because the right hand side is a power-sum of the
first kind with degree $|\lambda|$ and genus $g+\frac{1}{2}$, and it
suffices to show, that
for any monomial symmetric function $m_{\lambda}$ and any $g$ we have
\begin{multline}
\label{sym:Q}
\bigg(\sum_{\mu} t^{|\mu|} m_{\lambda}(\mu) (2g - 1)^{l(\mu)}
(-1)^{l(\mu)} z^{-1}_{\mu} p_{\mu} \bigg) \bigg(\sum_{\rho} t^{|\rho|}
(-1)^{l(\rho)} z^{-1}_{\rho} p_{\rho} \bigg) = \\
\bigg(\sum_{\nu} t^{|\nu|} \left(\frac{2g - 1}{2g}\right)^{l(\lambda)}
m_{\lambda}(\nu) (2g)^{l(\nu)} (-1)^{l(\nu)} z^{-1}_{\nu}
p_{\nu} \bigg),
\end{multline}
because the right hand side is a power-sum of the
second kind with degree $|\lambda|$ and genus $g+\frac{1}{2}$.
In order to prove \eqref{sym:R} and \eqref{sym:Q} it is enough to use Lemma \ref{lem:technical1}.
\end{proof}

The main result of this subsection is the following proposition.
\begin{proposition}
\label{prop:it-is-nice}
 $\frac{P_{\lambda}(t)}{C(t)}$ is a linear combination of power-sums of the
first (respectively, second) kind of genus
$\frac{|\lambda|}{2}$ and degree at most $|\lambda|-l(\lambda)$.
\end{proposition}
\begin{proof}
It is enough to apply Lemma \ref{lem:admissible} and Lemma \ref{lem:nice}.
%  and Lemma \ref{case2}.
\end{proof}

\subsection{Goulden-Rattan formula}

\label{sec:Goulden-Rattan}
For a partition $\lambda$ let $m_{\lambda}$ denote the monomial symmetric
function in
variables $x_1,x_2,\dots$. In
this paper we consider the particular evaluation of the monomial
symmetric function at $x_i=i$, for $i=1,\dots ,k-1$, and $x_i=0$,
for $i\geq k$, and write this as $\hat{m}_{\lambda}$. Let $A(t)$ be a formal
power series. We denote the coefficient of $t^k$ in $A(t)$ by $[t^k]A(t)$.

\begin{theorem}[Goulden and Rattan \cite{GouldenRattan2007}]
\label{theo:g-r}
For $g\geq 1$, $k\geq 2g-1$,
\begin{equation}
\label{eq:G-R}
\Sigma_{k,k+1-2g}=-\frac{1}{k}[t^{k+1-2g}]
\sum_{|\lambda| = 2g}\hat{m}_{\lambda}
\frac{P_{\lambda}(t)}{C(t)}.
\end{equation}
\end{theorem}

% There is a slight modification
% of this result, given below, in which the term corresponding to
% the partition with one part is given a simpler (but equivalent)
% evaluation.
%
% \begin{theorem}\label{mainmod}
% For $n\geq 1$, $k\geq 2n-1$,
% \begin{equation*}
% \Sigma_{k,2n}=-\frac{1}{k}[t^{k+1-2n}]
% \left(\frac{k-1}{2n} \hat{m}_{2n}P_{2n-1}(t)+\sum_{\stackrel{\lambda\vdash
% 2n}{l(\lambda)\geq 2}}?_{\lambda}
% \frac{P_{\lambda}(t)}{C(t)}\right).
% \end{equation*}
% \end{theorem}

% The following result gives a generating function form of the main result.
%
% \begin{theorem}\label{maingen}
% For $n\geq 1$, $k\geq 2n-1$,
% \begin{equation*}
% \Sigma_{k,2n}=-\frac{1}{k}[u^{2n}t^{k+1}]\frac{1}{C(t)}\prod_{j=1}^{k-1}
% \left(1+\sum_{i\geq 1}j^iP_i(t)u^it^i\right),
% \end{equation*}
% \begin{equation*}
% \Sigma_k=-\frac{1}{k}[t^{k+1}]\frac{1}{C(t)}\prod_{j=1}^{k-1}
% \left(1+\sum_{i\geq 1}j^iP_i(t)t^i\right).
% \end{equation*}
% \end{theorem}
%
% Note that, for each $n\geq 1$, these results
% give $\Sigma_{k,2n}$ as the coefficient of $t^{k+1-2n}$ in
% a polynomial in $C(t)$ and
% $$D^iC(t)=\sum_{m\geq 2} m^iC_mt^m,\;\;\;\; i\geq 1.$$
% Thus $\Sigma_{k,2n}$ is written as
% a polynomial in the $C_m$'s, with coefficients that are
% rational in $k$, so our results give $C$-expansions for
%  $\Sigma_{k,2n}$, for $n\geq 1$.
%

\subsection{Proof of Proposition \ref{prop:main-prop}}
\begin{proof}[Proof of Proposition \ref{prop:main-prop}]
Equation \eqref{eq:G-R} can be written in the form
$$k\
\Sigma_{k,k+1-2g}=-[t^{k+1-2g}]
\sum_{|\lambda| = 2g}\hat{m}_{\lambda}
\frac{P_{\lambda}(t)}{C(t)}.$$

The evaluation of the power sum symmetric function 
$$ \hat{p}_s = 1^s + \cdots + (k-1)^s$$
analogous to that for
$\hat{m}_\lambda$ is a polynomial in $k$ of degree $s+1$; it follows immediately
that the
evaluation of the power-sum symmetric function $\hat{p}_\lambda$ is a
polynomial in $k$ of degree $|\lambda|+l(\lambda)$. The monomial symmetric
function $m_{\lambda}$ is a linear combination of power-sum symmetric functions
$p_{\mu}$, where each partition $\mu$ which appears in this linear combination
is obtained from partition $\lambda$ by gluing some of their parts (see for
example \cite{Macdonald1995}). It
means that for each such $\mu$ we have
$$|\lambda|+l(\lambda) \geq  |\mu|+l(\mu),$$
and for this reason also 
$\hat{m}_\lambda$ is a polynomial in $k$ of degree at
most $|\lambda|+l(\lambda)$.

For any partition $\mu$ such that $|\mu| = k + 1-  2g $ we have $k =
p_1(\mu) + 2g -1$, where $p_1$ is a power symmetric function, hence
there exists a symmetric function $f_{\lambda}$ of degree
$|\lambda|+l(\lambda)$ which does not depend on $k$ such that $\hat{m}_{\lambda}
= f_{\lambda}(\mu)$.
Proposition \ref{prop:it-is-nice}
finishes the proof.
\end{proof}

\section{Divisibility of polynomials}
\label{sec:divisibility}

\subsection{Implications of divisibility}

At this step we proved Proposition \ref{prop:main-prop}. In order to
prove Theorem \ref{theo:main} we would like to show that for each integer $g
\geq 1$, functions $\tilde{f}_g$ and $\tilde{h}_g$ are divisible by the
symmetric function
$$(p_1 + 2g -1)\frac{(p_1 + 2g)(p_1 + 2g -1)(p_1 + 2g
-2)}{3!},$$ 
where $p_1$ denotes the power symmetric function. By word divisible
we mean that there exist symmetric functions $f_g$ and $h_\mu$
such that 
$$ \tilde{f}_g = (p_1 + 2g -1)\frac{(p_1 + 2g)(p_1 + 2g -1)(p_1 + 2g
-2)}{3!}f_g$$
and
$$ \tilde{h}_g = (p_1 + 2g -1)\frac{(p_1 + 2g)(p_1 + 2g -1)(p_1 + 2g
-2)}{3!}h_g.$$
Observe that for fixed $g \geq 1$ and for any partition $\mu$ there exists
number $k$ such that $|\mu| = k + 1 - 2g$ and then
$$ \tilde{f}_g(\mu) = k \binom{k+1}{3}f_g(\mu),$$
$$ \tilde{h}_g(\mu) = k \binom{k+1}{3}h_g(\mu)(\mu).$$
The idea of showing divisibility of symmetric function by symmetric function of
degree $1$ is similar to the case of showing divisibility of some polynomial by
some monomial. The main idea is dividing with a remainder and using a fact that
if some integer is divisible by infinite number of primes then it has to be
equal to zero. The remainning of this section is a formalisation of this idea.

The first difficulty is that we have to deal with polynomials in several
variables. Hence, in order to show some generalisation of the ``dividing with
remainder'' technique, we need the following technical lemma:

\begin{lemma}
\label{lem:zerozero}
Let $m$ be a fixed integer and $f$ be a polynomial in variables $x_k,\dots,x_l$
with the property that
$ f(\mu_k,\dots,\mu_l)=0$ for all integers $\mu_k,\dots,\mu_l\geq 1$ which
fulfill the following equations:
\begin{equation}
\label{eq:suma}
\sum_{k \leq j \leq l}\mu_j > m,
\end{equation}
\begin{equation}
\label{eq:wzrost}
\mu_i>\mu_{i+1}+\cdots+\mu_l
\end{equation}
for all values of $i$ for which it makes sense. Then $f=0$.
\end{lemma}

The proof of this lemma can be found in Section \ref{sec:technical}.

The next lemma is key for this section: it allows to translate information
about arithmetic properties of Kerov polynomials into information about the
polynomials governing the coefficients.
\begin{lemma}
\label{lem:sie-dzieli}
Let $f$, respectively $h$, be a symmetric function of degree at most $d$ with
rational coefficients, $g\geq 1$ be an integer; we define
% \begin{equation}
\begin{align*} L_{k} =  & \sum_{|\mu|=k+1-2g}
\big(l(\mu) + 2g - 2\big)!\ f(\mu)\ \mathcal{R}_\mu, \\
\intertext{respectively,}
L'_{k}&= \sum_{|\mu|=k+1-2g} \big(2g - 1\big)^{l(\mu)}\ h(\mu)\ \mathcal{Q}_\mu
\end{align*}
and view it as a polynomial in $R_2,R_3,\dots$.
% \end{equation}

Assume that an integer $\Delta$ has a property that all coefficients of
$L_{p+\Delta}$ (respectively, all coefficients of $L'_{p+\Delta}$) are integers
divisible by $p$ for an infinite number of prime numbers $p$. Then there exists
a symmetric function $\tilde{f}$ (respectively, $\tilde{h}$) with rational
coefficients of  degree at most $d-1$ such that
\begin{align*} L_{k} &= (k-\Delta) \sum_{|\mu|=k+1-2g}
\big(l(\mu) + 2g - 2\big)!\ \tilde{f}(\mu)\ \mathcal{R}_\mu, \\
\intertext{respectively,} 
L'_k&= (k-\Delta) \sum_{|\mu|=k+1-2g}
\big(2g - 1\big)^{l(\mu)}\ \tilde{h}(\mu)\ \mathcal{Q}_\mu.
\end{align*}
\end{lemma}

\begin{proof}
For simplicity assume that the coefficients of $f$ (respectively, $h$)
are integer numbers; if this is not the case we multiply $L_{k}$ and $f$
(respectively, $L'_k$ and $h$) by some common multiple of the denominators.

Let $\mu=(\mu_1,\mu_2,\dots)$ be a sequence of indeterminates. We use the
notation $|\mu|=\mu_1+\mu_2+\cdots$ and define variable
$z=|\mu|+2g-1-\Delta$. The family of indeterminates $\mu$ can be alternatively
parametrized by $z,\mu_2,\mu_3,\dots$; we just use the substitution
$\mu_1=z+\Delta+1-2g-\mu_2-\mu_3-\cdots$.  Now we can consider $f,h \in
\Lambda[z]$ as polynomials in one variable $z$ with coefficients in the ring
$\Lambda$ of symmetric functions in variables $\mu_2,\mu_3,\dots$
and we can divide $f$ and $h$ by $z$ with
a remainder. Hence
\begin{align*} f(\mu)&= (|\mu|+2g-1-\Delta) \tilde{f}(\mu) +
r(\mu_2,\mu_3,\dots), \\ 
h(\mu)&= (|\mu|+2g-1-\Delta) \tilde{h}(\mu) +
s(\mu_2,\mu_3,\dots)  
\end{align*}
for some $\tilde{f}, \tilde{h} \in \Lambda[z]$ and for some $r, s \in \Lambda$.
Below we will show that $r=s=0$. This would imply that
\begin{align*} f(\mu)&= (|\mu|+2g-1-\Delta) \tilde{f}(\mu), \\
h(\mu)&= (|\mu|+2g-1-\Delta) \tilde{h}(\mu),
\end{align*}
where by substitution $z=|\mu|+2g-1-\Delta$ we view $\tilde{f},\tilde{h}\in
\Lambda[\mu_1]$ as polynomials in one variable $\mu_1$ with coefficients in
$\Lambda$. For any permutation $\pi$ of the set of positive integers which
moves only finitely many elements we have
\begin{multline*} (|\mu|+2g-1-\Delta) \tilde{f}(\mu_1,\mu_2,\dots) =
f(\mu_1,\mu_2,\dots) =\\ f(\mu_{\pi(1)},\mu_{\pi(2)},\dots) =
(|\mu|+2g-1-\Delta) \tilde{f}(\mu_{\pi(1)},\mu_{\pi(2)},\dots)
\end{multline*}
hence from the cancellation property
$$ \tilde{f}(\mu_1,\mu_2,\dots) = \tilde{f}(\mu_{\pi(1)},\mu_{\pi(2)},\dots)$$
is a symmetric function. In an analogous way we show that $\tilde{h}$ is a
symmetric function. The lemma follows now immediately.

It remains now to show that $r=s=0$. From the following on let $\mu_2> \cdots >
\mu_l$ be fixed integers bigger than
$1$ which fulfill Equations \eqref{eq:wzrost} and \eqref{eq:suma} with $m =
\Delta - 2g$. 
% Firstly, we will show that 
% $$r(\mu_2, \dots, \mu_l,0, \dots)=0$$
% by showing that $r(\mu_2, \dots, \mu_l,0, \dots)$ is divisible by infinite many
% primes. Secondly, we notice that $r$ will fullfill assumptions of Lemma
% \ref{lem:zerozero} which shows that
% $r(x_2, \dots, x_l$, $0, \dots) = 0$ as a polynomial in indeterminates $x_2,
% \dots, x_l$ for any $l>1$ hence $r=0$, as claimed. In the case of the function
% $s$
% the technique is almost the same, but details are more delicate.
Define
$\mu_1=p+\Delta+1-2g-\mu_2-\mu_3-\cdots-\mu_l$,
where $p$ is a prime number. Notice that $\mu_1 - 1 < p$, because we required
that $\Delta+1-2g-\mu_2-\mu_3-\cdots-\mu_l < 1$. We consider the integral vector
$\mu=(\mu_1,\dots,\mu_l)$.

If $p$ is large enough, the parts of $\mu$ are all distinct and it follows
that
$$[R_{\mu_1} R_{\mu_2} \cdots R_{\mu_l}] L_{p+\Delta} =   (l+2g-2)!\ f(\mu)\
(\mu_1-1)(\mu_2-1)\cdots(\mu_l-1).$$
Also, if prime number $p$ is big enough then it does not divide
$(l+2g-2)!\ (\mu_1-1)(\mu_2-1)\cdots(\mu_l-1)$. It follows that
for infinitely
many prime numbers $p$ the number
$$ f(\mu)=  p \tilde{f}(\mu) + r(\mu_2, \dots, \mu_l)  $$
is divisible by $p$. We proved in this way that $r(\mu_2, \dots, \mu_l,
0,
\dots)$ is an
integer which is divisible by an infinite number of primes hence
$r(\mu_2, \dots, \mu_l, 0, \dots)=0$. Finaly Lemma \ref{lem:zerozero}
shows
that $r=0$.

If $p$ is big enough then condition
\eqref{eq:wzrost} holds true for all $1\leq i \leq l-1$ therefore every
partition resulting from $\mu$ by gluing together some of its parts cannot be
obtained by gluing the parts of $\mu$ in some other way (in fact, this
property is the main reason of introducing Equation \eqref{eq:wzrost}). 
From \eqref{eq:qi} it follows that
$$[R_{\mu_1} R_{\mu_2} \cdots] L'_{p+\Delta} = 
\sum_{\nu\geq\mu}(2g-1)^{l(\nu)}\big(l(\nu)-1\big)!\ h(\nu),$$
where $\nu\geq\mu$ means that partition $\nu$ can be obtained from partition
$\mu$ by gluing some parts of $\mu$.
% More formally, for partitions $\nu = (\nu_1,\dots, \nu_k), \mu=(\mu_1,\dots,
% \mu_l), \nu\geq\mu$ iff $\nu_1 = \mu_{\sigma(1)}+\cdots+\mu_{\sigma(i_1)},
% \nu_2
% = 
% \mu_{\sigma(i_1+1)}+\cdots+\mu_{\sigma(i_1+i_2)},\dots, \nu_k =
% \mu_{\sigma(i_1+\cdots+i_{k-1}+1)}+\cdots+\mu_{\sigma(l)}, $ for some $\sigma
% \in \Sym(l).$
% \textbf{Mam nadzieje, ze ta definicja jest poprawna. Waruek, ktory wczesniej
% byl bledny zapewnia, ze nie da sie skleic na 2 rozne sposoby.}
We also know that for infinitely many prime numbers $p$ the following number
\begin{multline*}
\sum_{\nu\geq\mu}(2g-1)^{l(\nu)}\big(l(\nu)-1\big)!\ h(\nu) = \\
p\left[ \sum_{\nu\geq\mu}(2g-1)^{l(\nu)}\big(l(\nu)-1\big)!\
\tilde{h}(\nu)\right] 
+ \sum_{\nu\geq\mu}(2g-1)^{l(\nu)}\big(l(\nu)-1\big)!\ s(\nu')
\end{multline*}
is divisible by $p$, where for $\nu
= (\nu_1,\nu_2,\dots)$ we denote $\nu' = (\nu_2,\nu_3,\dots)$. Notice that the
set of values of $\nu'$ which contribute to the right
hand side does not depend on the choice of $p$, because only $\nu_1$ depends on
the choice of $p$.
Thus we proved that for infinitely many prime numbers $p$ the second
summand on the
right hand side does not depend on the choice of $p$ and 
is a fixed integer divisible
by all these prime numbers, hence 
\begin{equation}
\label{eq:upper-trangular}
 \sum_{\nu\geq\mu}(2g-1)^{l(\nu)}\big(l(\nu)-1\big)!\ s(\nu') = 0.  
\end{equation}

We will use induction over $k$ to show that $s(x_2,\dots,x_k,0,\dots)$ is equal
to the zero polynomial for any $k > 1$, which proves that $s=0$. Indeed, assume,
that
$s(x_2,\dots,x_k,0,\dots)$ is equal
to the zero polynomial for $k < l$. From
the induction hypothesis it follows that all summands
on the left-hand side of \eqref{eq:upper-trangular} vanish, except for
$\nu=\mu$, which shows that $s(\mu')=0$. We use Lemma
\ref{lem:zerozero} to show that $s=0$ as claimed.
\end{proof}

\subsection{Divisibility}

In order to prove Theorem \ref{theo:main} we would like to apply Lemma
\ref{lem:sie-dzieli} to Kerov polynomials. For this, we need some interesting
arithmetic properties of coefficients of Kerov polynomials. The next lemma,
which was formulated as a conjecture by Światosław Gal \cite{Gal}, shows some
properties of these kind.

\begin{lemma}
\label{lem:div}
If $p$ is an odd prime number then 
\begin{enumerate}[label=(\alph*)]
 \item \label{lem:gal_a} $ \frac{\Sigma_{p}-R_{p+1}+2 R_2}{p}$,
 \item \label{lem:gal_b} $ \frac{\Sigma_{p-1}-R_p}{p} $,
 \item \label{lem:gal_c} $\frac{\Sigma_{p+1}-R_{p+2}+R_3}{p}$
\end{enumerate}
are polynomials in free cumulants $R_2,R_3,\dots$
with nonnegative integer coefficients.
\end{lemma}

We will prove this Lemma in Section \ref{sec:technical}, because the proof is
very technical. Finally, we can prove the main result.

\subsection{Proof of the main result}

\begin{proof}[Proof of Theorem \ref{theo:main}]
We know by Proposition \ref{prop:main-prop} that for any integer $g\geq 1$ there
exist inhomogeneous symmetric
functions $\tilde{f}_g$ and $\tilde{h}_g$, having maximal degree $4g$, such
that
\begin{align*}
L_k := L'_k := k\ K_{k,k+1-2g} &=  \sum_{|\mu|=k+1-2g}
\big(l(\mu) + 2g - 2\big)! \ \tilde{f}_g(\mu) \mathcal{R}_\mu \\ 
 &= \sum_{|\mu|=k+1-2g}
(2g - 1)^{l(\mu)}\ \tilde{h}_g(\mu) \mathcal{Q}_\mu.
\end{align*}
By applying Lemma \ref{lem:sie-dzieli} for $\Delta=0$ we obtain that
\begin{align*}
k\ K_{k,k+1-2g} &=  k\ \sum_{|\mu|=k+1-2g}
\big(l(\mu) + 2g - 2\big)! \ \tilde{f}'_g(\mu) \mathcal{R}_\mu \\ 
 &= k\ \sum_{|\mu|=k+1-2g}
(2g - 1)^{l(\mu)}\ \tilde{h'}_g(\mu) \mathcal{Q}_\mu,
\end{align*}
where $\tilde{f}'_g, \tilde{h}'_g$ are symmetric functions of degree at most
$4g-1$.

Let 
\begin{align*}
% \label{eq:pierwsza_podzielnosc}
L_k := L'_k := K_{k,k+1-2g} &=  \sum_{|\mu|=k+1-2g}
\big(l(\mu) + 2g - 2\big)! \ \tilde{f}'_g(\mu) \mathcal{R}_\mu \\ 
 &= \sum_{|\mu|=k+1-2g}
(2g - 1)^{l(\mu)}\ \tilde{h}'_g(\mu) \mathcal{Q}_\mu.
\end{align*}
Lemma \ref{lem:div}\ref{lem:gal_a} shows that Lemma \ref{lem:sie-dzieli} can be applied for
$\Delta=0$, thus
\begin{align*}
K_{k,k+1-2g} &=  k\ \sum_{|\mu|=k+1-2g}
\big(l(\mu) + 2g - 2\big)! \ \tilde{f}''_g(\mu) \mathcal{R}_\mu \\ 
 &= k\ \sum_{|\mu|=k+1-2g}
(2g - 1)^{l(\mu)}\ \tilde{h}''_g(\mu) \mathcal{Q}_\mu,
\end{align*}
where $\tilde{f}''_g, \tilde{h}''_g$ are symmetric functions of degree at most
$4g-2$.

Let 
\begin{align*}
L_k := L'_k := \frac{A K_{k,k+1-2g}}{k} &=  \sum_{|\mu|=k+1-2g}
\big(l(\mu) + 2g - 2\big)! A\ \tilde{f}''_g(\mu) \mathcal{R}_\mu \\ 
 &= \sum_{|\mu|=k+1-2g}
(2g - 1)^{l(\mu)}A\ \tilde{h}''_g(\mu) \mathcal{Q}_\mu,
\end{align*} 
where $A$ is the common multiple of the denominators of coefficients of
$\tilde{h}''_g$ and $\tilde{f}''_g$. We know that $L_{p-1} = L'_{p-1}$ has
integer coefficients as polynomial in $R_2,R_3,\dots$ and we know, thanks to
Lemma \ref{lem:div}\ref{lem:gal_b}, that for infinitely many prime
numbers $p$ the coefficients of $(p-1)L_{p-1} = (p-1)L'_{p-1}$ are divisible by $p$,
hence coefficients of $L_{p-1} = L'_{p-1}$ are also divisible by $p$, because
$p-1$ and $p$ are coprime. Then we can apply Lemma
\ref{lem:sie-dzieli} for
$\Delta=-1$ and we obtain that there exist symmetric functions $\tilde{f}'''_g,
\tilde{h}'''_g$
of degree at most $4g-3$ such that
\begin{align*}
K_{k,k+1-2g} &=  k\ (k+1)\ \sum_{|\mu|=k+1-2g}
\big(l(\mu) + 2g - 2\big)! \ \tilde{f}'''_g(\mu) \mathcal{R}_\mu \\ 
 &= k\ (k+1)\ \sum_{|\mu|=k+1-2g}
(2g - 1)^{l(\mu)}\ \tilde{h}'''_g(\mu) \mathcal{Q}_\mu,
\end{align*}
where $\tilde{f}''_g, \tilde{h}''_g$ are symmetric functions of degree at most
$4g-3$.

Similarly as before, thanks to Lemma \ref{lem:div}\ref{lem:gal_c} and thanks
to the fact that for prime number $p>2$ the numbers $p$ and $p+1$ are coprime and the
numbers $p+2$ and $p$ are coprime, we can apply Lemma \ref{lem:sie-dzieli} for
$\Delta=1$ for
\begin{align*}
L_k := L'_k := \frac{B K_{k,k+1-2g}}{k(k+1)} &=  \sum_{|\mu|=k+1-2g}
\big(l(\mu) + 2g - 2\big)! B\ \tilde{f}'''_g(\mu) \mathcal{R}_\mu \\ 
 &= \sum_{|\mu|=k+1-2g}
(2g - 1)^{l(\mu)}B\ \tilde{h}'''_g(\mu) \mathcal{Q}_\mu,
\end{align*} 
where $B$ is the common multiple of the denominators of coefficients of
$\tilde{h}'''_g$ and $\tilde{f}'''_g$ and we obtain that there exist symmetric
functions $\tilde{f}''''_g,
\tilde{h}''''_g$
of degree at most $4g-4$ such that
\begin{align*}
K_{k,k+1-2g} &= (k-1)\ k\ (k+1)\ \sum_{|\mu|=k+1-2g}
\big(l(\mu) + 2g - 2\big)! \ \tilde{f}''''_g(\mu) \mathcal{R}_\mu \\ 
 &= (k-1)\ k\ (k+1)\ \sum_{|\mu|=k+1-2g}
(2g - 1)^{l(\mu)}\ \tilde{h}''''_g(\mu) \mathcal{Q}_\mu,
\end{align*}
which finishes the proof.
% \begin{align*}
% \tilde{f}'_g = (p_1 + 2g -1) \tilde{f}''_g = (p_1 + 2g) \tilde{f}'''_g = (p_1
% + 2g -2) \tilde{f}''''_g
% \end{align*}
% and
% \begin{align*}
% \tilde{h}'_g = (p_1 + 2g -1) \tilde{h}''_g = (p_1 + 2g) \tilde{h}'''_g = (p_1
% + 2g -2) \tilde{h}''''_g.
% \end{align*}
% This shows that there exist symmetric functions $f_g, h_g$ of degree at most
% $4g-4$ such that
% $$\tilde{f}'_g = \frac{(p_1 + 2g)(p_1 + 2g -1)(p_1
% + 2g -2)}{3!}f_g$$
% and
% $$\tilde{h}'_g = \frac{(p_1 + 2g)(p_1 + 2g -1)(p_1
% + 2g -2)}{3!}h_g,$$
% where $p_1$ is a power symmetric function.
% Putting it into \eqref{eq:pierwsza_podzielnosc} we obtain that
% \begin{multline}
% K_{k,k+1-2g} =  \sum_{|\mu|=k+1-2g} \binom{p_1(\mu)+2g}{3}  
% \big(l(\mu) + 2g - 2\big)!\ f_g(\mu)\ \mathcal{R}_\mu \\
%  =   \sum_{|\mu|=k+1-2g} \binom{p_1(\mu)+2g}{3}  
% (2g - 1)^{l(\mu)}\ h_g(\mu)\ \mathcal{Q}_\mu \\
%  =  \binom{k+1}{3}  \sum_{|\mu|=k+1-2g}
% \big(l(\mu) + 2g - 2\big)!\ f_g(\mu)\ \mathcal{R}_\mu \\ 
%  =  \binom{k+1}{3}  \sum_{|\mu|=k+1-2g}
% (2g - 1)^{l(\mu)}\ h_g(\mu)\ \mathcal{Q}_\mu,
% \end{multline}
% which finishes the proof.
\end{proof}

\section{Technical lemmas}
\label{sec:technical}

In this Section we prove all technical lemmas we used in this article.

% \begin{lemma}
%  \label{lem:technical1}
% The following abstract equalities hold:
% \begin{equation}
% \label{sym:R'}
%  \sum_{\mu \cup \rho = \nu} m_{\lambda}(\mu) \frac{(l(\mu) + 2g -2)!}{l(\mu)!}
% u_{\mu} u_{\rho} = \frac{m_{\lambda}(\nu)}{l(\lambda) + 2g -1} \frac{(l(\nu) +
% 2g - 2)!}{l(\nu)!} u_{\nu};
% \end{equation}
% \begin{equation}
% \label{sym:Q'}
%  \sum_{\mu \cup \rho = \nu} m_{\lambda}(\mu) (2g -1)^{l(\mu)}
% z^{-1}_{\mu} z^{-1}_{\rho} = \left(\frac{2g-1}{2g}\right)^{l(\lambda)}
% m_{\lambda}(\nu)(2g)^{l(\nu)} z^{-1}_{\nu}.
% \end{equation}
% \end{lemma}

\subsection{Identities on symmetric functions}
\begin{lemma}
 \label{lem:technical1}
The following abstract equalities hold:
\begin{equation}
\label{sym:R'}
 \sum_{\mu \cup \rho = \nu} m_{\lambda}(\mu) \frac{(l(\mu) + 2g -2)!}{l(\mu)!}
u_{\mu} u_{\rho} = \frac{m_{\lambda}(\nu)}{l(\lambda) + 2g -1} \frac{(l(\nu) +
2g - 2)!}{l(\nu)!} u_{\nu};
\end{equation}
\begin{equation}
\label{sym:Q'}
 \sum_{\mu \cup \rho = \nu} m_{\lambda}(\mu) (2g -1)^{l(\mu)}
z^{-1}_{\mu} z^{-1}_{\rho} = \left(\frac{2g-1}{2g}\right)^{l(\lambda)}
m_{\lambda}(\nu)(2g)^{l(\nu)} z^{-1}_{\nu}.
\end{equation}
\end{lemma}
% We postpone the technical proof to Section \ref{sec:technical}.
% \ref{sec:technical}.
\begin{proof}
% [Proof of Lemma \ref{lem:technical1}]
From the definition of the monomial symmetric function we know that
$m_{\lambda}(\mu) = 0$ for all partitions $\mu$ such that $l(\mu) < l(\lambda)$.
We use an identity that for every integer $n$ such that $l(\lambda)
\leq n \leq l(\nu)$ we have
\begin{equation}
\label{sym:mon}
\binom{l(\nu) - l(\lambda)}{n - l(\lambda)} m_{\lambda}(\nu) = 
\sum_{\substack{\mu \cup
\rho = \nu, \\ l(\mu) = n}} m_{\lambda}(\mu)
\left( \prod_{i \geq 1} \frac{m_i(\nu)!}{m_i(\mu)!\ m_i(\rho)!}\right).
\end{equation}
Indeed, 
\begin{multline*}
\sum_{\substack{\mu \cup
\rho = \nu, \\ l(\mu) = n}} m_{\lambda}(\mu)
\left( \prod_{i \geq 1} \frac{m_i(\nu)!}{m_i(\mu)!\ m_i(\rho)!}\right) =
\sum_{\substack{\mu \subset \nu, \\ l(\mu) = n}} m_{\lambda}(\mu) = \\
\sum_{\substack{\mu' \subset \mu \subset \nu, \\ l(\mu) = n, l(\mu') = 
l(\lambda)}} m_{\lambda}(\mu') = 
\sum_{\substack{\mu' \subset \nu, \\ l(\mu') = l(\lambda)}} \sum_{\substack{\mu'
\subset \mu \subset \nu, \\ l(\mu) = n}} m_{\lambda}(\mu') = \\
\binom{l(\nu) -
l(\lambda)}{n - l(\lambda)} \sum_{\substack{\mu' \subset \nu, \\ l(\mu') =
l(\lambda)}} m_{\lambda}(\mu') = \binom{l(\nu) - l(\lambda)}{n - l(\lambda)}
m_{\lambda}(\nu),
\end{multline*}
where  $\sum_{\substack{\mu \subset \nu, \\ l(\mu) = n}}$ means that $\mu =
(\nu_{\sigma(1)}, \dots, \nu_{\sigma(n)})$ for some $\sigma \in \Sym(l(\nu))$
such that $\sigma(i) < \sigma(i+1)$ for all $i \in
\{1,\dots,n-1\}$ and we are summing over all such permutations $\sigma$.
Now, we can write the left hand side of \eqref{sym:R'} in the
following way:
\begin{multline*}
 \sum_{l(\lambda) \leq n \leq l(\nu)} \frac{(n + 2g -2)!}{n!}
\Bigg( \sum_{\substack{\mu \cup
\rho = \nu,\\ l(\mu) = n}} m_{\lambda}(\mu) \prod_{i \geq 1}
\frac{m_i(\nu)!}{m_i(\mu)\
!m_i(\rho)!}\Bigg) \times \\ \shoveright{u_{\nu} \frac{n!\ (l(\nu) -
n)!}{l(\nu)!} =} 
\\ 
\shoveleft{ \sum_{l(\lambda) \leq n \leq l(\nu)} \frac{(n + 2g -2)!}{n!}
\frac{(l(\nu) - l(\lambda))!}{(l(\nu) - n)!(n - l(\lambda))!}
m_{\lambda}(\nu) \times }
\\ 
\shoveright{ u_{\nu} \frac{n!\ (l(\nu) -
n)!}{l(\nu)!} = }\\ \frac{(l(\nu) - l(\lambda))!}{l(\nu)!} \sum_{l(\lambda) \leq
n
\leq l(\nu)}
\frac{(n + 2g -2)!}{(n - l(\lambda))!} m_{\lambda}(\nu) u_{\nu}
\end{multline*}
and using the equality 
$$
 \sum_{0 \leq i \leq b} \binom{a+i}{i} = \binom{a+b+1}{b}$$
we have
\begin{multline*}
\frac{(l(\nu) - l(\lambda))!}{l(\nu)!} \sum_{l(\lambda) \leq n \leq l(\nu)}
\frac{(n + 2g -2)!}{(n - l(\lambda))!} = \\ 
\frac{(l(\nu) -
l(\lambda))!(l(\lambda) + 2g -2)!}{l(\nu)!} \sum_{0 \leq n \leq l(\nu) -
l(\lambda)}
\frac{(n + l(\lambda) + 2g -2)!}{n!(l(\lambda) + 2g -2)!} = \\
\frac{(l(\nu) -
l(\lambda))!(l(\lambda) + 2g -2)!}{l(\nu)!} \sum_{0 \leq n \leq l(\nu) -
l(\lambda)} \binom{l(\lambda) + 2g -2 + n}{n} = \\
\frac{(l(\nu) -
l(\lambda))!\ (l(\lambda) + 2g -2)!\ (l(\nu) + 2g -1)!}{l(\nu)!\ (l(\nu) -
l(\lambda))!\ (l(\lambda) + 2g -1)!} = \frac{(l(\nu) + 2g
-1)!}{l(\nu)!(l(\lambda) + 2g -1)}
\end{multline*}
which finishes the proof of \eqref{sym:R'}.

Using  \eqref{sym:mon} we can write the left hand side of 
\eqref{sym:Q'} in the following form:
\begin{multline*}
 \sum_{l(\lambda) \leq n \leq l(\nu)} (2g -1)^{n}
\Bigg( \sum_{\substack{\mu \cup
\rho = \nu, \\l(\mu) = n}} m_{\lambda}(\mu) \prod_{i \geq 1}
\frac{m_i(\nu)!}{m_i(\mu)\ !m_i(\rho)!}\Bigg) z^{-1}_{\nu} = \\
\Bigg(\sum_{l(\lambda) \leq n \leq l(\nu)} (2g -1)^{n} \binom{l(\nu) -
l(\lambda)}{n - l(\lambda)}\Bigg) m_{\lambda}(\nu) z^{-1}_{\nu} = \\
(2g - 1)^{l(\lambda)} \Bigg(\sum_{0 \leq n \leq l(\nu) - l(\lambda)}
\binom{l(\nu) - l(\lambda)}{n}(2g -1)^n \Bigg) m_{\lambda}(\nu) z^{-1}_{\nu} =
\\
\left(\frac{2g-1}{2g}\right)^{l(\lambda)}
m_{\lambda}(\nu)(2g)^{l(\nu)} z^{-1}_{\nu},
\end{multline*}
where the last equality holds because of the binomial identity:
$$
\sum_{0 \leq n \leq m}\binom{m}{n}a^n = (a + 1)^m,
$$
which finishes the proof. 
\end{proof}

% \begin{lemma}
% \label{lem:zerozero}
% Let $f$ be a polynomial in variables $x_k,\dots,x_l$ with a property that
% $ f(\mu_k,\dots,\mu_l)=0$ for all integers $\mu_k,\dots,\mu_l\geq 1$ which
% fulfill equation \eqref{eq:wzrost}. Then $f=0$.
% \end{lemma}

\subsection{Proof of Lemma \ref{lem:zerozero}}
\begin{proof}[Proof of Lemma \ref{lem:zerozero}]
Let $l$ be fixed; we will use backward induction over $k$. For $k=l$
we know that $f(\mu_k)=0$ for infinitely many choices of $\mu_k$. In other
words, polynomial $f$ has infinitely many zeros hence $f=0$, as claimed.

Let us assume that the inductive assertion holds for some $k\leq l$.
We can write
$$f(x_{k-1}, \dots, x_{l})
= \sum_{0 \leq i \leq N} x_{k-1}^{i}\ f_i(x_k, \dots, x_{l})$$
for some $N$. Let us fix integers $\mu_{k}, \dots, \mu_{l}$ bigger than $1$
which satisfy
\eqref{eq:wzrost} and \eqref{eq:suma}. Then we can find infinitely many
integer
numbers $\mu_{k-1}$ for which the vector $(\mu_{k-1},\dots,\mu_{l})$
satisfies both \eqref{eq:wzrost} and \eqref{eq:suma}; for each such a number we
have
$f(\mu_{k-1}, \dots, \mu_{l}) = 0$ therefore the polynomial $x_{k-1}\mapsto
f(x_{k-1},\mu_{k},\dots,\mu_{l})$ has infinitely many zeros hence it is the
zero polynomial and 
$f_i(\mu_{k}, \dots, \mu_{l})=0$. This shows that the inductive assertion
can be applied to the polynomial $f_i(x_{k},\dots,x_{l})$ and therefore
$f_i(x_{k},\dots,x_{l})=0$. This finishes the proof.
\end{proof}

\subsection{Arithmetic properties of Kerov polynomials}

\subsubsection{Auxiliary results}

We present two theorems we need to prove Lemma \ref{lem:div}.

\begin{theorem}[Dołęga, F\'eray, \'Sniady \cite{DolegaF'eray'Sniady2008}]
\label{theo:dfs}
Let $k\geq 1$ and let $s_2,s_3,\dots$ be a sequence of non-negative integers
with only finitely
many non-zero elements. The coefficient of $R_2^{s_2} R_3^{s_3} \cdots $ in
the Kerov polynomial $K_{k}$ is equal to the number of triples
$(\sigma_1,\sigma_2,q)$ with the following properties:
\begin{enumerate}[label=(\alph*)]
 \item \label{enum:first-condition}
$\sigma_1,\sigma_2$ is a factorization of the
cycle; in other words $\sigma_1,\sigma_2\in \Sym(k)$ are such that $\sigma_1
\circ \sigma_2=(1,2,\dots,k)$;
 \item \label{enum:ilosc2} the number of cycles of $\sigma_2$ is equal to the
number of\/ factors in
the product $R_2^{s_2} R_3^{s_3} \cdots $; in other words
$|C(\sigma_2)|=s_2+s_3+\cdots$;
 \item \label{enum:boys-and-girls} the total number of cycles of $\sigma_1$ and
$\sigma_2$ is equal to the degree of the product
$R_2^{s_2} R_3^{s_3} \cdots $; in other words
$|C(\sigma_1)|+|C(\sigma_2)|=2 s_2+3 s_3+4 s_4+\cdots$;
 \item \label{enum:kolorowanie} $q:C(\sigma_2)\rightarrow \{2,3,\dots\}$ is a
coloring of the cycles of
$\sigma_2$ with a property that each color $i\in\{2,3,\dots\}$ is used exactly
$s_i$ times (informally, we can think that $q$ is a map which to cycles of
$C(\sigma_2)$ associates the factors in the product $R_2^{s_2} R_3^{s_3}
\cdots$);
% we require that for every
% color $i\in\{2,3,\dots\}$ there are exactly $s_i$ cycles of $\sigma_2$
% with color $i$;
 \item \label{enum:marriage} for every set $A\subset C(\sigma_2)$ which is
nontrivial (i.e., $A\neq\emptyset$ and $A\neq C(\sigma_2)$) there are more than
$\sum_{i\in A} \big( q(i)-1 \big) $ cycles of\/ $\sigma_1$ which intersect
$\bigcup A$.
\end{enumerate}
\end{theorem}

We say that a partition $\Pi$ of the set $[k]=\{1,\dots,k\}$ is a \emph{pushing partition} if any pair of neighboring
elements of $[k]$ with respect to the cyclic order (i.e.~$i$ and $i+1$ are a pair of neighboring elements for
any $1\leq i\leq k-1$ as well as $1$ and $k$) does not belong to the same block of $\Pi$.

The cyclic group $\Z/ k \Z$ acts on the set of all partitions (respectively, the set of pushing partitions) of the set
$[k]$ as follows: for a partition $\Pi$ of $[k]$ and $i\in \Z/ k \Z$  we define $i + \Pi$ as the partition of $[k]$
with a property that $a,b$ belong to the same block of $\Pi$ if and only if $a',b'$ belong to the same block of $i +
\Pi$ for all $a,b,a',b'\in [k]$ such that $a+i \equiv a' \pmod k$, $b+i \equiv b' \pmod k$.

For any pushing partition $\Pi$ it is possible (see \cite{'Sniady2006a}) to define the normalized character
$\Sigma_\Pi$. It has a property that $\Sigma_\Pi=\Sigma_\pi$, where the right-hand side should
be understood as in \eqref{eq:character} for some $\pi\in\Sym(l)$, $l\geq 1$. So defined partition-indexed character
has the following properties:
\begin{theorem}[Proposition 4.4, Claim 3.1, Proposition 3.2 in \cite{'Sniady2006a}]
\label{theo:discrete-math}
\
\begin{itemize}
 \item The map $\Pi\mapsto\Sigma_\Pi$ is constant on the orbits of the action of the cyclic group $\Z/ k \Z$ on
the set of pushing partitions of $[k]$.
   \item For any integer $k\geq 2$
\begin{equation} 
  \label{eq:pushing}
  R_k = \sum_{\Pi} I_\Pi\ \Sigma_\Pi ,
\end{equation}
where the sum runs over pushing partitions of $[k]$ and $I_\Pi\in \Z$, called free index, is constant on the orbits of
the action of the cyclic group $\Z/ k \Z$ on the set of pushing partitions of $[k]$.
 \item For the minimal partition $\Pi=\big\{ \{1\}, \dots, \{k\} \big\}$ the corresponding character is given by
     $$ \Sigma_{ \{ \{1\}, \dots, \{k\} \} } = \Sigma_{k-1}. $$
\end{itemize}

\end{theorem}

% \begin{lemma}
% \label{lem:div}
% If $p$ is an odd prime number then 
% \begin{enumerate}[label=(\alph*)]
%  \item \label{lem:gal_a} $ \frac{\Sigma_{p}-R_{p+1}+2 R_2}{p}$,
%  \item \label{lem:gal_b} $ \frac{\Sigma_{p-1}-R_p}{p} $,
%  \item \label{lem:gal_c} $\frac{\Sigma_{p+1}-R_{p+2}+R_3}{p}$
% \end{enumerate}
% are polynomials in free cumulants $R_2,R_3,\dots$
% with nonnegative integer coefficients.
% \end{lemma}

\subsubsection{Proof of Lemma \ref{lem:div}}

\begin{proof}[Proof of Lemma \ref{lem:div}]

In the following we shall prove that the coefficients are integer numbers.
Their nonnegativity would follow from Theorem \ref{theo:dfs}.

In order to prove that the coefficients of $\frac{\Sigma_{p}-R_{p+1}+2 R_2}{p}$
are integer we consider the action of the cyclic group $\mathbb{Z}/p\mathbb{Z}$
on the set of triples $(\sigma_1,\sigma_2,q)$ which contribute to Theorem
\ref{theo:dfs} defined by conjugation $$\psi(i) (\sigma_1,\sigma_2,q) = \big(
c^i \sigma_1 c^{-i}, c^i \sigma_2 c^{-i}, q'),$$ where $c=(1,2,\dots,k)$ is the
cycle and $q'(a) = q(c^{-i} a c^i)$ for $a \in C(\sigma_2)$.
%  we leave the
% details how to define $q'$ as a simple exercise. 
All orbits
of this action consist of $p$ elements except for the fixpoints of this action
which are of the form $\sigma_1=c^a$, $\sigma_2=c^{1-a}$. These fixpoints
contribute to the monomial $R_{p+1}$ (with multiplicity $1$) and to the monomial
$R_2$ (with multiplicity $p-2$).  This finishes the proof of the integrality of
coefficients of \ref{lem:gal_a}.

We apply Theorem \ref{theo:discrete-math} in the case when $k=p$ is a prime number.
The right-hand side of \eqref{eq:pushing} is constant on each orbit of the action of the cyclic group $\Z / p \Z$. Each
orbit of this action consists of $p$ elements, except for the fixpoints. The
only
pushing partition of $[p]$ which is invariant under the action of the cyclic group $\Z / p \Z$ is the minimal partition 
$\big\{ \{1\}, \dots, \{p\} \big\}$. In this way we proved that 
\begin{multline*}  R_p = \Sigma_{p-1} + p\ \big( \text{linear combination of the characters $\Sigma_\pi$}
\\ 
\text{for $\pi\in\Sym(l)$, $l\geq 1$ with integer coefficients} \big). \end{multline*}
Thanks to Kerov polynomials, each $\Sigma_\pi$ can be written as a polynomial in free cumulants with integer
coefficients. This shows part \ref{lem:gal_b}.

In the following we shall use the notations and results presented in the paper
of Biane  \cite{Biane2003}. In order to prove part \ref{lem:gal_c} we consider
the formal power
series 
$$ H(z) = z-\sum_{j\geq 1} {B}_{j+1}z^{-j}$$ 
where ${B}_{j}$ are Boolean cumulants. Biane showed that
\begin{equation}
\label{eq:hahaha}
(-p-1)\Sigma_{p+1} =
[z^{-1}]H(z)H(z-1)\cdots H(z-p)
\end{equation}
and 
$$(-p-1) {R}_{p+2} =  [z^{-1}]H(z)^{p+1}.$$
We know from \cite{Biane2003} that $B_j$ is a polynomial in free cumulants
$R_2,R_3,\dots$ with integer coefficients as well as $\Sigma_{p+1}$ is a
polynomial in Boolean cumulants $B_2,B_3,\dots$ with integer coefficients; hence
it suffices to show that $(-p-1)(\Sigma_{p+1}-R_{p+2}+R_3)$ is a polynomial in
Boolean cumulants with all coefficients divisible by $p$. It is equivalent to
show that $\Sigma_{p+1}-R_{p+2}+R_3=0$ under additional assumption
that all coefficients of the power series are taken from a field of
characteristic $p$, hence all formulas are considered in a field of
characteristic $p$ from now. 

From \eqref{eq:hahaha} it follows that 
$$ [B_3] \Sigma_{p+1} = \frac{1}{p+1} \sum_{0\leq z \leq p}
\frac{1}{2} \frac{d^2}{dz^2} 
\big[ z (z-1) \cdots (z-p) \big]. $$
From Fermat's little theorem (see for example \cite{GrahamKnuthPatashnik}) it
follows that in the field of characteristic $p$
$$ z (z-1) \cdots (z-p) = z (z^p-z)=z^{p+1}-z^2 $$
hence
\begin{equation}
\label{eq:B3}
 [B_3] \Sigma_{p+1} = \frac{1}{p+1} \sum_{0\leq z \leq p} (-1)=-1. 
\end{equation}

We define ${B}_0 = -1$ and ${B}_1 = 0$; then using binomial formula we have
$$ 
H(z-i)=\sum_{j\geq-1}\sum_{k\geq 0} 
(-1)^{k+1} \binom{-j}{k} i^k {B}_{j+1}  z^{-(j+k)}, $$
hence
\begin{multline}
\label{eq:iloczynn}
-\frac{1}{p+1} H(z)H(z-1)\cdots H(z-p) = 
% \\ 
% \sum_{j_0 \geq -1}(-{B}_{j_0 + 1})z^{-j_0} \times \\ 
% \sum_{j_1 \geq
% -1}(-{B}_{j_1+1})\sum_{k_1 \geq 0}\binom{-j_1}{k_1}(-1)^{k_1}z^{-(j_1+k_1)}
% \times \\ \sum_{j_2 \geq -1}(-{B}_{j_2+1}) \sum_{k_2 \geq
% 0}\binom{-j_2}{k_2}(-2)^{k_2}z^{-(j_2+k_2)} \times \\ \cdots\sum_{j_p \geq
% -1}(-{B}_{j_p+1})\sum_{k_p \geq 0}\binom{-j_p}{k_p}(-p)^{k_p}z^{-(j_p+k_p)} =
\\ -\frac{1}{p+1} \sum_{k \in A} \sum_{j \in B}
(-1)^{|k|_1+p+1} 
\\
\left(\prod_{0 \leq i \leq
p}\binom{-j_i}{k_i}i^{k_i}{B}_{j_i+1}\right)
z^{-|j|_0-|k|_1},
\end{multline}
where $A,B \subset \Z^{p+1}$ such that 
$$A = \{(k_0,k_1,\dots,k_p): k_i \geq 0\
\text{for}\ 0 \leq i \leq p\},$$ 
$$B = \{(j_0, j_1,\dots,j_p): j_i \geq -1\
\text{for}\ 0 \leq i \leq p\}$$ 
and for $k = (k_0, k_1, \dots, k_p)$, $i \in \{0,1\}$ the sum
$k_i + k_{i+1} +\cdots + k_p$ is denoted by $|k|_i$.
 For any $a\in
\Z/p\Z$ such that $a\neq 0$ the map $x\mapsto ax$ is a bijection of the multiset
$(0,1,\dots,p)\subset\Z/p\Z$ (notice that $0=p$ appears twice in this multiset)
therefore the left-hand side of \eqref{eq:iloczynn} is equal to 
\begin{multline}
\label{eq:iloczynn2}
-\frac{1}{p+1} H(z)H(z-a)\cdots H(z-pa) = 
% \\ 
% \sum_{j_0 \geq -1}(-{B}_{j_0 + 1})z^{-j_0} \times \\ 
% \sum_{j_1 \geq
% -1}(-{B}_{j_1+1})\sum_{k_1 \geq 0}\binom{-j_1}{k_1}(-1)^{k_1}z^{-(j_1+k_1)}
% \times \\ \sum_{j_2 \geq -1}(-{B}_{j_2+1}) \sum_{k_2 \geq
% 0}\binom{-j_2}{k_2}(-2)^{k_2}z^{-(j_2+k_2)} \times \\ \cdots\sum_{j_p \geq
% -1}(-{B}_{j_p+1})\sum_{k_p \geq 0}\binom{-j_p}{k_p}(-p)^{k_p}z^{-(j_p+k_p)} =
\\ -\frac{1}{p+1} \sum_{k \in A} \sum_{j \in B}
(-1)^{|k|_1+p+1} a^{|k|_1} 
\\
\left(\prod_{0 \leq i \leq
p}\binom{-j_i}{k_i}i^{k_i}{B}_{j_i+1}\right)
z^{-|j|_0-|k|_1}.
\end{multline}

The coefficient of $z^{-1}$ in \eqref{eq:iloczynn2} can be viewed as a
polynomial in $a$; we shall denote it by $P(a)$. In the following we will study
its coefficients of highest degrees. We are interested only in the
summands for which $|j|_0+|k|_1=1$; since $|j|_0\geq
-p-1$ therefore $|k|_1\leq p+2$ and the degree of $P(a)$ is at
most $p+2$. 

However, $|k|_1=p+2$ would correspond to the case
$j=(-1,\-1,\dots,-1)$ which is equivalent to setting $B_2=B_3=\cdots=0$;
therefore
$[a^{p+2}] P(a)=0$. 

For $|k|_1=p+1$ there is no summand for which
$j_0,\dots,j_p\neq 0$ hence $[a^{p+1}] P(a)=0$. 

For $|k|_1=p$ every
summand which contributes is of the following form: one of the numbers
$j_0,\dots,j_p$ is equal to $1$ and all the others are equal to $-1$. This 
shows that $[a^p] P(a)$ viewed as a polynomial in $B_2,B_3,\dots$ contains only
one monomial, namely a multiple of $B_2$; also $[B_2] P(a)$ viewed as a
polynomial in $a$ contains only one monomial namely a multiple of $a^p$. 
Therefore $[a^{p}]P(a)$ is a multiple of $B_2$ and the value of the
coefficient of $B_2$ fulfills:
$$[B_2][a^{p}]P(a)=[B_2] P(1)=[B_2] \Sigma_{p+1}. $$ 
Since $p$ is odd, the expansion of $\Sigma_{p+1}$ into Boolean cumulants
contains only summands which are of odd degree \cite{Biane2003}; it follows
that $[a^{p}]P(a)=0$.

In an analogous way we prove that $[a^{p-1}]P(a)$ is a multiple of $B_3$ and 
$$[B_3][a^{p-1}]P(a)=[B_3] \Sigma_{p+1}= -1 $$  
from \eqref{eq:B3}.

In this way we proved that $P(a)$ is a polynomial of degree $p-1$ which
takes the same value for all $a\in\{1,\dots,p-1\}$. Polynomial
$$\widetilde{P}(a)=  -B_3 a^{p-1} + P(0) $$
has the same properties. It follows that $P-\widetilde{P}$ has degree at most
$p-2$ which takes the same value for all $a\in\{1,\dots,p-1\}$ hence it must be
equal to the constant. It follows that $\widetilde{P}=P$.

Therefore 
$$ \Sigma_{p+1} = P(1) = -B_3 +  R_{p+2}. $$
Observation that $B_3=R_3$ finishes the proof for the third expression.
\end{proof}

% We leave it as an exercise to check that the proof for the last expression
% can be adapted to cover the cases of the first two expressions as well. 
It is
interesting that for the first two expressions we managed to find combinatorial
proofs while for the last expression there seems to be no natural candidate
for a combinatorial approach. 

\section*{Acknowledgments}

% Research of P{\'S} is supported by the MNiSW research grant P03A 013 30,
% by the EC Marie Curie Host Fellowship for the Transfer of Knowledge
% ``Harmonic Analysis, Nonlinear Analysis and Probability'', contract
% MTKD-CT-2004-013389 and by
% by 7010 POLONIUM project ``Non-Com\-mu\-ta\-tive Harmonic Analysis with
% Applications
% to Operator Spaces, Operator Algebras and Probability''.

Research is supported by the Polish Ministry of Higher Education research
grant N
N201 364436 for the years 2009--2012.

PŚ thanks Marek Bożejko, Philippe Biane,
Akihito Hora, Jonathan Novak,
Światosław Gal and Jan Dymara for several stimulating discussions during various
stages of this research project.

\bibliographystyle{alpha}

\bibliography{biblio2009}

\end{document}